\documentclass[reqno,11pt]{amsart}
\usepackage{amscd,amssymb,verbatim,array}
\usepackage{hyperref}
\usepackage{amsmath}
\usepackage[active]{srcltx} 
\usepackage{tikz-cd}
\usepackage[all]{xy}

\numberwithin{equation}{section} \theoremstyle{plain}

\newtheorem{theorem}{Theorem}[section]
\newtheorem{lemma}[theorem]{Lemma}
\newtheorem{proposition}[theorem]{Proposition}

\theoremstyle{definition}

\theoremstyle{remark}

\numberwithin{equation}{section}

\newcommand{\grad}{\operatorname{grad}}
\newcommand{\img} {\operatorname{img}}

\newcommand{\ind}{\operatorname{ind}}

\begin{document}
\title {Witten deformation and the spectral package 
of a Riemannian manifold}

\author{Dan Burghelea}

\keywords{Witten deformation, Spectral package}

\subjclass{35P20}

\date{\today}



\begin{abstract}
The Witten deformation associated to a Morse function $f$ on a closed Riemannian manifold $(M,g)$ via Rellich-Kato theorem (see section 3) relates analytically the spectral package of $(M,g)$ (eigenvalues and eigenforms) to the Morse complex defined by $(g,f)$ coupled with the "multivariable harmonic oscillators" associated to the critical points of $f.$ We survey this relation and discuss some implications, including the  finite subset of the spectral package referred to as the "virtually small spectral package" .
\end{abstract}

\dedicatory{ Dedicated to the memory of Cabiria Andreian Cazacu  
a mathematician I much admired}
\maketitle

\setcounter{tocdepth}{1}
\tableofcontents

\section {Introduction}

A smooth map $f: M\to \mathbb R$ on a  closed Riemannian manifold $(M,g),$  via Witten deformation, provides an analytic family of  self adjoint operators 
$\Delta^q(t)$ with  $\Delta^q(0)=\Delta^q$  the Laplace - Beltrami operator on  $q-$ differential forms on $M$,  see section 3 below, for definition. This family is  "holomorphic  of type A"  in the sense of Kato, cf \cite{TK}.

Rellich-Kato theorem , cf \cite {TK} Theorem 3.9,  organizes the eigenvalues and the eigenforms of $\Delta^q(t),$ for all $t\in \mathbb R,$ as  a  {\it countable collection}   $\mathcal A^q$ of analytic 
real-valued functions $\lambda^q_\alpha (t) $ and analytic $q-$differential form-valued maps $\omega^q_\alpha(t)$ s.t. for any $t$ \ $\Delta^q(t)( \omega^q_\alpha(t))= \lambda^q_\alpha(t)\omega^q_\alpha(t)$ and the collection $\{\lambda^q_\alpha(t), \omega^q_\alpha (t), \alpha\in \mathcal A^q\}$ represent the entire spectral package of $\Delta^q(t) $  (i.e. the collection of eigenvalues and eigenvectors of an operator, in this case $\Delta^q(t)$).  We call $\lambda^q_\alpha(t)$s and $\omega ^q_\alpha (t)$s branches  of eigenvalues  and eigenforms. 

If $f$ is a Morse function then the infinite collection $\mathcal A^q$ clusters as a disjoint union of finite collections $\mathcal A^q(k),$ $k\in \mathbb Z_{\geq 0},$ 
and the  set $\mathcal A^q(\infty)$ conjecturally empty, characterized by the 
property $\lim_{t\to \infty} \lambda ^q_\alpha(t) /t = 2 c_k$, $0= c_0 <c_1 <c_2<\cdots ,$ $c_k\to \infty$,  whose cardinalities for each $k$ is finite and depends only on the set of critical points $x$ 
and the scalar product $g(x),$ the metric $g$ provides on $T_x(M)$ 
\footnote {One 
can ask what happens when $t\to -\infty.$ The answer is similar; if one writes $\Delta^{q,f}(t)$ instead $\Delta^q(t)$ in order to indicate what smooth function is considered, then the Hodge operator $\star$ intertwines $\Delta^{q,f}(t)$ with $\Delta^{n-q, -f}(t)= \Delta^{n-q,f} (-t).$  This makes $\lambda^{q,f}_{\cdots}(-t)= \lambda^{n-q, f}_{\cdots}(t).$
}. 
In particular if in the small neighborhood of each critical point the metric is flat then the cardinality depends only on the number of critical points of each index, as specified by Theorem \ref{T3.3}. 

The  cluster corresponding to $k=0$  is in bijective correspondence to the set of critical points of index $q$  
and determines  an analytic family in $t\in \mathbb R$ of cochain complexes with scalar product,  all with the components and the cohomology of the same dimension. 
The cochain complex corresponding to $t=0$  
is isometrically embedded in $(\Omega^\ast (M), d^\ast) $ equipped with  the scalar product defined by the  metric $g,$ and is left invariant by $\Delta^q. $   The cochain complex corresponding to $t,$ when $t\to \infty,$ is $O(1/t)-$isometric to the appropriately scaled 
 geometric complex  (Morse complex) defined  by the vector field $(-\grad_g f)$ and equipped  with the scalar product  making the  characteristic functions of  critical points orthonormal, cf. Theorem \ref{T6.1}.

The restriction to $t=0$ selects from the infinite spectral package of $(M,g)$  a finite collection $\{\lambda^q_\alpha(0), \omega^q_\alpha(0),$ $\alpha \in \mathcal A^q(0)\subset \mathcal A^q\}$
referred below as the {\it $q-$virtually small spectral package} of $(M,g)$ associated with $f$ which has remarkable topological and geometrical properties, cf.Theorem \ref {T6.2}.  

It is important to note that the collection of virtually small eigenvalues which is of cardinality $c_q=\sharp  Cr_q(f)$  does not necessary consists of the smallest $c_q$ eigenvalues of $\Delta^q,$  however they   are the restriction at $t=0$ of the  eigenvalue branches $\lambda_\alpha^q(t)$ whose restriction to $t$ large enough exhaust the first smallest $c_q$  
eigenvalues of $\Delta^q(t).$ In different words, they become, by analytic continuation in $t,$ the smallest $c_q$ eigenvalues.

Note also that, in view of  analyticity,  the restriction of $\omega_\alpha^q(t)$  to an arbitrary small neighborhood of $t_0$ 
determines completely $\omega_\alpha^q(t)$ for any $t$ and then  $\omega^q_\alpha(0).$  

 {\bf Observation:} In view of  Propositions \ref{P5.2} and \ref{P5.1}, the restrictions of $\omega^q_\alpha (t_0)$ when  $\alpha\in \mathcal A^q(0)$ and  $\alpha\in \mathcal A^q\setminus \mathcal A^q(0)$  for $t_0$  very large are very different, as one can read off from Proposition (\ref {P5.2}) and Proposition (\ref{P5.1}) respectively;  one expects this be manifest in some way for the restriction at $t=0$ and provide a way to recognize the collection $\mathcal A^q(0).$ 

\vskip .1in
The purpose of this largely expository paper is to present these observations as a starting point  in exploring the geometric and topological applications/implications  of the virtually small spectral package.

Theorem \ref {T6.2} item (2)  suggests that at least one topological invariant  involving the  entire spectral package, as stated in Theorem \ref{T2.1} item (2), can be recovered from the virtually small spectral package. One expects that this is the case with many more geometric and topological invariants,     

Theorems \ref {T3.3} and \ref{T6.2} can be regarded as new statements,  while the others, Theorems \ref{T3.2}, \ref {T4.1}, \ref{T6.1} only as  some improvements of results already published.

The paper contains in addition to introduction  six more sections. 
 
-   Section 2  reviews  the {\it spectral package} of a closed Riemannian manifold.

-   Section 3 reviews the Witten deformation and  Witten-Laplace operators $\Delta^q(t)$, the Rellich - Kato  result on analytic branches of eigenvalues and 
eigenforms for $\Delta^q(t)$ and their clustering.

-   Section 4 reviews the CW structure defined by a pair $(g,f)$  
and the associated geometric complex known as the Morse complex. 

- Section 5 sketches the proof of Theorems \ref{T3.2} and \ref{T3.3}.

- Section 6 discusses the cluster $\mathcal A^q(0)$ and sketches the proof of Theorems \ref{T6.1} and \ref{T6.2}.

- Section 7  formulates a few conjectures  and  problems. 

\section {Spectral package of a Riemannian manifold $(M^n,g)$} \label {S:specpac}

Let $M^n$ be a closed differential manifold equipped with a Riemannian metric  $g$.
The differential structure on $M$ provides

\begin{itemize}
\item  $\Omega^r(M^n)$ - the vector space of differential forms of degree $r,$
\item $\mathcal X(M)$ - the vector space  of smooth vector fields,  

both modules over $\Omega^0(M)= \mathcal D(M),$
the commutative algebra of smooth functions, 
\item $\wedge : \Omega^r(M) \times \Omega^{r'}(M) \to \Omega^{r+r'}(M)$ 
- the exterior multiplication,  $\Omega^0(M)-$ bilinear map, 
\item  $\iota^r_X: \Omega^\ast(M)\to \Omega^{\ast-1}(M)$ 
 - the contraction w.r. to $X\in \mathcal X(M),$ $\Omega^0(M)-$ linear map,
 \item  $d^r:\Omega^r (M)\to \Omega^{r +1}(M) $  
- the exterior differential, a differential operator of order one,
\item  $L^r_X: \Omega^r(M)\to \Omega^r(M)$  
- the Lie derivative along $X$, $X\in \mathcal X(M)$ defined by $L^r_X:= d^{r-1}\cdot \iota^r_X - \iota^{r+1}_X\cdot d^r$ 
a differential operator of order one.
\end{itemize}
   
The pair $(\Omega^\ast, d^\ast), \ast=0,1,\cdots$ is a cochain complex, known as the de-Rham complex, with  de-Rham cohomology 
$\mathcal H^r_{DR}(M):= \ker d^r/ \img d^{r-1}.$ 
The de-Rham  theorem establishes a canonical isomorphism from  de-Rham cohomology $\mathcal H^r_{DR}(M)$ to the singular  cohomology with real coefficients of $M,$ isomorphism  obtained via integration of $r-$differential forms on singular smooth simplexes. 

\vskip .1in

For simplicity we will assume $M$ orientable. 
The  Riemannian metric $g$ on $M$ provides
the Hodge star-operator  $\star^r: \Omega^r(M^n)\to \Omega^{n-r}(M)$ with $\star^{n-r}\cdot \star^r= (-1)^{r(n-r)}$ and then the scalar product 
\begin{equation} \label {E1}\
\langle \omega, \omega'\rangle  :=\int _M \omega \wedge \star^r \omega'
\end{equation}
on each $\Omega^r(M).$ 
This in turn provides  the differential operator 
$\delta^r:\Omega^r(M^n)\to \Omega^{r-1}(M),$ the formal adjoint of $d^{r-1},$ defined by $\delta^r:= (-1)^{n(r-1)+1} \star^{n-r+1} \cdot  d^{n-r} \cdot \star^r$ and 
then the operator $\Delta^r:\Omega^r(M)\to \Omega^r(M).$  This is a formally self adjoint, nonnegative definite differential  operator of order two defined by 
$$\Delta^r:= \delta^{r+1}\cdot d^r + d^{r-1}\cdot \delta^r$$ called  the 
$r-$Laplace - Beltrami operator or simpler  the $r-$Laplacian.

Note that $$\star^q \cdot \Delta^q= \Delta^{n-q}\cdot \star^{n-q}.$$ 

Note also that one has   following commutative up to sign diagram 

\begin{equation}\label {E2.2}
\xymatrix{
\cdots\ar[r]&\Omega^q\ar[r]^{d^q} \ar[d]_{\star ^q}&\Omega^{q+1}\ar[r]^{d^{q+1}}\ar[d]_{\star^{q+1}} &\Omega^{q+2}\ar[r]^{d^{q+2}}\ar[d]_{\star^{q+2}} &\cdots\\ 
\cdots\ar[r]&\Omega^{n-q}\ar[r]^{\delta^{n-q}} &\Omega^{n-q-1}\ar[r]^{\delta^{n-q-1}} &\Omega^{n-q-2}\ar[r]^{\ \delta^{n-q-2}} &\cdots}
\end{equation} 
with 
$$\star^{q+1}\cdot d^q=  (-1)^{q+1}\delta^{n-q}\cdot \star^q . $$

Since 
the homology of the second row of the diagram  (\ref {E2.2}) is  canonically identified via the metric $g$ to the de-Rham homology of $M$ based on currents,  
 the $\star-$operator realizes the familiar Poincar\'e duality from $H_{DR}^r(M)$ to $H^{DR}_{n-r}(M).$

Since $M^n$ is a closed manifold then 
{\it formally adjoint} becomes  {\it adjoint} on $L_2(\Omega^q(M))$ the $L_2-$closure of $\Omega^q(M)$ with respect to the scalar product (\ref{E1}). 

Each $\Delta^r$ has spectrum consisting of the infinite sequence of eigenvalues, nonnegative real numbers $ \lambda^r_k,$ \ 
$0= \lambda_0 ^r  < \lambda ^r_1 < \cdots  < \lambda_k^r   < \lambda ^r_{k+1}< \cdots,$ \ 
each with finite multiplicity, which increase to $\infty.$ 

Denote by  $H^r_k\subset \Omega^r(M)$  the eigenspace corresponding  to the eigenvalues $\lambda^k_r,$ $H^r_k :=\ker (\Delta^r-\lambda^r_k Id)\subset \Omega^r(M)$ and note that $H^r_k\perp H^r_{k'}$ for $k\ne k'.$   

Let $H^r_{0,\mathbb Z} \subset H^r_0 \subset \Omega^r (M)$ be the lattice of  {\it integral harmonic $r-$forms}  i.e harmonic forms which take integer values when integrated on $\mathbb Z-$cycles of dimension $r,$ and let $T_r: =H^r_0/ H^r_{0,\mathbb Z}$  be the compact torus equipped with the obvious Riemannian metric induced from the scalar product 
$\langle\  ,\ \rangle$ restricted to $H^r_0.$   

Denote by $V^r :=vol_g T_r.$  Note that $V^0=1, V^n=vol_g  M^n$ and let $$\mathbb V(M,g):=\prod_i (V^i)^{(-1)^i}.$$

%
%
%

The table  
\begin{equation*}
\begin{aligned}
0 =      &\lambda_0 ^r  <   &\lambda ^r_1 <   &\cdots  <  &\lambda_k^r   < \ \ &\lambda ^r_{k+1}<  &\cdots \\ 
           &H_0^r              \    &H^r_1            \   & \cdots   \   &H^r_k              \  \ \ &H^r_{k+1}         \      &\cdots
\end{aligned}           
\end{equation*} 
with $H^r_i:=  \ker (\Delta^r- \lambda^r_i Id)$, 
 is referred to as the {\it spectral package} of $(M,g).$  Note that the spectral package corresponding to $q$ identifies via $\star^q-$operator to the spectral package  corresponding to $(n-q).$  
 
 Sometimes $H^r_i$ is specified by an orthonormal base.  
 Sometimes the table {\it spectral package}  might  contain the numbers  $V^r$'s. 
 Any finite part of the spectral package is, in principle, computable with arbitrary accuracy but not the entire spectral package,.

Hodge theorem establishes a canonical decomposition  $$\Omega^r(M)=\Omega^r_+(M) \oplus \Omega^r_-(M)\oplus H^r_0$$ in mutually orthogonal subspaces 
  $\Omega^r(M)_+=\img d^{r-1},$  $\Omega^r(M)_-=\img \delta^{r+1}$ and $H^r_0= \ker \Delta^r= \ker d^r\cap \ker \delta^r,$ decomposition which diagonalizes $\Delta^r$  and satisfies $\star^r: \Omega^r(M)_\pm\to \Omega^{n-r}(M)_\mp$ 
and which implies the decomposition $H^r_k= H^r_{k,+}\oplus  H^r_{k,-}$ with $H^r_{k,\pm}= H^r_k\cap \Omega^r(M)_{\pm}.$
This shows that a Riemannian metric realizes de-Rham cohomology canonically as a subspace of differential forms referred to as {\it harmonic forms} and one has the morphisms of co-chain complexes 
$$ \xymatrix {(H_0^\ast, 0)\ar[r]^{in}& (\Omega^\ast(M), d^\ast)\ar[r]^{pr}& (H^\ast_0, 0)}$$ with composition the identity.

In view of the ellipticity and nonnegativity of $\Delta^r$ one has the well defined positive real numbers, the zeta-regularized determinants   $\det \Delta^r_{\pm}, \det' \Delta^r,$ which represent the   $\zeta-$regularized product of the nonzero eigenvalues of $\Delta^r_{\pm},  \Delta^r,$ 
\footnote { $\Delta^r_{\pm}= \Delta^r  |_{\Omega^r_{\pm}(M)},$} cf \cite {Cheeger} 
and then 

$$T_{an}(M,g):= \prod_r     (\sqrt {{\det}' \Delta^r}\ )^{(-1)^{r+1}  r}   = \prod_r  (\sqrt{{\det } \Delta_+^r}\  )^{(-1)^{r+1}}$$

referred to as the {\it analytic torsion}.   
   
Recall that for any compact ANR $X,$  in particular for any compact manifold, the integral homology $H_r(X;\mathbb bZ)$ is a finitely generated abelian group of a finite rank $\beta_r$ whose set of finite order elements have finite cardinality $\mathbb Tor_i.$ 
The following two numbers are remarkable topological invariants 

$$\chi(X):=\sum (-1)^i \beta_i(X)\  \text {and} \ \mathbb Tor(X):=\prod (\mathbb Tor_i(X))^{ (-1)^i}     .$$

%
%

Note that when $X=M$ is a closed odd dimensional manifold $\chi(M)=0$  and
when $X$ is an oriented even dimensional manifold $\mathbb Tor(M)=1.$  
 
 The following two familiar results are among the many relations between topology and the spectral packages.
 
 \begin{theorem}\label {T2.1} \  
 \begin{enumerate} 
 \item (de-Rham  -  Hodge)  $\chi(M)= \sum_i (-1)^i \dim H^i_0$
 \item (Cheeger Muller) $\ln \mathbb Tor (M)=- \sum _i (-1)^{i} \ln V^i + \ln T_{an} (M,g)$ 
 \end{enumerate} 
 \end{theorem}
 Item (1) derives the integer $\chi(M)$  from the multiplicity of the eigenvalue $0$ of $\Delta^r$s while item (2) derives the integer $\mathbb Tor(M)$ from the nonzero eigenvalues of  $\Delta^r$s cf. 
\cite {Cheeger} Theorem 8.35.
 
\section  {Witten Deformation and the spectral package of a triple $(M,g,f)$} \label {S;wittendef}

Let $f: M\to \mathbb R$  be a smooth function.
The {\it Witten deformation } 
is a one parameter family of scaled de-Rham complexes, precisely  the family of cochain complexes 
$(\Omega^\ast(M), d^\ast (t)), \  t\in \mathbb R,$  with the differential $$d^q(t)\omega = e^{-t f}d^q (e^{tf}\omega )=  d^q \omega +t df\wedge \omega, \ \omega\in \Omega^q(M).$$     
 The  Laplacians $ \Delta^{q}(t):= d^{q-1}(t)\cdot  \delta^q(t)  +  \delta^{q+1}(t) \cdot d^q(t)$ with  $\delta^q(t): = (-1) ^{n(q-1) +1}\star^{n-q+1} \cdot  \
 d^{n-q}(t) \cdot \star ^q$ referred below as the {\it Witten Laplacians}
can be expressed in terms of the Laplacians $\Delta^q$ as follows:
 \begin{equation} \label {E3.1}
\Delta^q(t)=  \Delta^q + t (L^q_X+ \mathcal L^q_X) + t^2||X||^2   .\end{equation}
 These operators remain self adjoint, nonnegative elliptic operators  on $L_2(\Omega^q(M)),$ the $L_2-$ completion of $\Omega^q(M).$  Here 
$\mathcal L_X^q$ its  formal adjoint of $L^q_X$ given by $\mathcal L_X^q:= (-1)^{(n+1) q+1} \star^{n-q} \cdot L^{n-q}_X \cdot \star^q .$   
  
  Hodge decomposition continues to hold for any $t,$  namely
  $$\Omega^q(M)=\Omega^r_+(M)(t) \oplus \Omega^r_-(M)(t)\oplus H^r_0(t)$$ 
with mutually orthogonal   
 $\Omega^r_+(M)(t)=\img d^{r-1}(t),$  $\Omega^r_-(M)(t)=\img \delta^{r+1}(t)$ and $H^r_0(t)= \ker \Delta^r(t)= \ker d^r(t)\cap \ker \delta^r(t),$ decomposition which diagonalizes $\Delta^r(t).$   
Note that in view of the isomorphism of $(\Omega^\ast (M), d^\ast)$ and $(\Omega^\ast (M), d^\ast (t))$  we have the isomorphism of their cohomology and then the equality 
\begin{equation}\label {E3.2}
\dim H^r_0= \dim H_0^r(t).
\end{equation}
Note that the diagram \ref{E2.2} can be enhanced  by replacing $d^\ast$ by $d^\ast (t)$ and   $\delta^\ast$ by $\delta^\ast (t).$

%

\vskip .1in
 
 As a consequence  of (\ref{E3.1}), by a result of Rellich - Kato, Theorem 3.9, chapter 7, in \cite {TK} one has Theorem.\ref{T3.1} below.
 \begin{theorem} (Rellich - Kato)\label {T3.1}\
 
 There exists a collection of nonnegative real-valued  functions  $\lambda^q_\alpha(t)$ unique up to permutation and a collection of  norm one $q-$differential form-valued maps  $\omega^q_\alpha(t) \in \Omega^q(M) ,$ analytic in $t\in \mathbb R,$  indexed by  $\alpha \in \mathcal A,$ $\mathcal A$ countable set, each with holomorphic extension to a neighborhood of the real line $\mathbb R\subset \mathbb C$ \footnote {holomorphic  extension means extensions $\lambda^q(z)\in \mathbb C$ , $\omega^q(z)\in \Omega(M)\otimes C$ for $z$ in a neighborhood of $\mathbb R$ in 
 $\mathbb C$ which for $t\in \mathbb R$ is a real number and $\omega^q(t)\in \Omega(M)\otimes 1$ } %
  s.t.
 \begin{enumerate}
 \item $\Delta^q(t)\omega^{q}_\alpha (t)= \lambda^{q}_\alpha (t)\omega^{q}_\alpha (t),$  
\item for any $t$ the collection $\lambda^{q}_\alpha (t)$ represent all repeated  eigenvalues of $\Delta^q(t)$  and the collection $\omega^{q}_\alpha (t)$ form a complete orthonormal family  of associated eigenvectors for the operator $\Delta^q(t),$  
\item exactly $\beta_q= \dim H^q(M;\mathbb R)$ eigenvalue functions $\lambda^q_\alpha(t)=0$ for any $t$, all other are strictly positive for any $t.$ 
\end{enumerate}
\end{theorem} 

The analytic functions $\lambda^q_\alpha(t)$ and the analytic differential form-valued maps $\omega^q_\alpha (t)$ are called {\it eigenvalue  branches}  and {\it eigenform branches}
and the collection $\{\lambda^q_\alpha (t), \omega^q_\alpha(t), \alpha\in \mathcal A\}$ is called the {\it spectral package  of $(M,g,f).$} 
\vskip .1in
We are interested in these branches when $f:M\to \mathbb R$ is a Morse function on $M.$  This means that for any $x\in Cr(f):=\{x\in M\mid df_x=0\}$ in some local chart $\varphi_x: U_x\to \mathbb R^n, U_x$ open neighborhood of $x,$ $\varphi_x$ open embedding with $\varphi_x(x)= (0,,0,\cdots 0)$ one has
\begin{equation}\label {E3.3}
f\cdot \varphi_x^{-1}(x_1,x_2,\cdots x_n)= f(x)-1/2 \sum_{1\leq i\leq k}  x_i^2 + 1/2 \sum_{k+1\leq i\leq n}x^2_i.
\end{equation}
In this case $x$ is called {\it critical point of Morse index $k$}.  Denote by $Cr_k(f)\subset Cr(f)$  the set of critical points of Morse index $k.$ A chart $\varphi$ with such property for a critical point is called {\it Morse chart} in which case the coordinates  $(x_1, x_2, \cdots x_n)$ are called {\it Morse coordinates}.

From now on, mostly for simplicity in estimates, we will assume  that  in a small neighborhood of  critical  points, there exists a   Morse chart  $\varphi$ s.t. the metric $g$ is given by $g_{i,j}(x_1, \cdots x_n)= \delta_{i,j}.$ Such metric  is called  {\it flat near critical points.}  We believe this assumption  can  be removed or at least weakened.
\vskip .1in

The following result  is credited  to E. Witten; a proof will be sketched in Section  \ref{S5} following close  \cite {BFKM} section 5 where a similar statement under more general hypotheses is treated.
 
\begin{theorem}\label {T3.2}
 Suppose $(M,g)$ is a closed Riemannian manifold and $f:M\to \mathbb R$ a Morse function with $c_q$ critical points of Morse index $q$.
 There exists constants $C_1, C_2, C_3$ and $T_0$ depending on $(M,g,f)$ s.t. for any $t>T_0$
 \begin{enumerate}
\item $Spect \Delta_q (t) \cap (C_1 e^{-C_2 t}, C_3 t)=\emptyset$ and $1\in (C_1 e^{-C_2 t}, C_3 t),$
\item for $t>T_0$ the number of eigenvalue branches  $\lambda^q_\alpha(t)$ with $\lambda^q_\alpha(t)\in [0, C_1 e^{-C_2t})$ 
counted with  multiplicity, is 
exactly  $c_q$  the number of critical points of Morse index $q,$
\item  any eigenform branch  $\omega^q_\alpha(t)$ with  eigenvalue branch $ \lambda^q_\alpha(t) <1$ for $t>T_0$  localizes at the set $Cr(f)_q,$ 
i.e. $\lim_{t\to \infty}  ||\omega^q_{\alpha}(t)||_{L_2 (M\setminus U)}= 0$ for any open neighborhood $U$ of  $Cr_q(f).$
\end{enumerate}
 \end{theorem}
 
 As a consequence 
 exactly $c_q$  eigenvalue  branches $\lambda^q_\alpha(t)$ go exponentially fast to $0$ when $t\to \infty,$ in particular 
 $\lim_{t\to \infty}\lambda^q_\alpha(t)/t=0$  and all other go more than  linearly fast to $\infty,$  in particular $\lim_{t\to \infty} \lambda^q_\alpha (t)= \infty.$   It can be shown that $\lambda^q_\alpha(t) \leq c_1 + c_2t + c_3 t^2$ cf \cite {SH}. Also, as a consequence, if $\lambda^q(t)$ is bounded for $t>0$ (say $\lambda^q(t)<1$) then $\lambda^q(t)$ goes exponentially fast to $0$ when $t\to\infty.$
 
 Denote by $\mathcal A^q_{vs}= \mathcal A^q(0):=\{\alpha\in A^q| \lim_{t\to \infty}  \lambda^q_\alpha(t)/t=0\}$ which is the same as the  set of  branches which go exponentially fast to $zero.$ These $c_q$ eigenvalue branches are called {\it virtually small} and the collection  $\{\lambda^q_\alpha (t), \omega^q_\alpha(t), \alpha\in \mathcal A\}$ the {\it virtually small spectral package of $(M,g,f).$}
    
Since $\Delta^q(-t)$ for $f$ is the same as $\Delta^q(t)$ for $-f,$  when $t\to -\infty$ exactly $\beta_{n-q}$ eigenvalue branches go exponentially fast to zero  the other more than linearly fast  to $\infty.$ As a consequence, as shown in a paper in preparation, if $\lambda^q(t)$ is bounbded then $\lambda^q(t)$ goes exponentially fast to $0$ when $t\to \pm \infty.$

 \begin{theorem}\label{T3.3}\

 If the metric $g$ is flat near the critical points of the Morse function $f$ then:
  \begin{enumerate}
 \item the set $\mathcal A^q\setminus \mathcal A^q(\infty)$ is in bijective correspondence with the collection of symbols 
 \newline $\{ \alpha= (x, I, P)\in Cr(f)\times \mathcal I^q \times \mathbb (Z_{\geq 0})^n \mid o(\alpha) \geq 0\}$ where
 
  $x\in Cr(f), I\in \mathcal I^q=\{(j_1, j_2,\cdots j_q)\mid 1\leq j_1 <j_2<\cdots j_q\leq n\},$
 
  $ P=(p_1,\cdots p_n), p_i\in \mathbb Z_{\geq 0}$ and
 
  $o(\alpha)= \sum_i p_i + q +index (x) -2N_{index (x)} (I)$  with $N_k(I):= \sharp\{j\leq k\mid j\in I\}.$ 
 
\item  each $\lambda^q_\alpha (t)$ satisfies  $$\lim _{t\to \infty} \lambda^q_\alpha (t)/ t= 2 o(\alpha)$$ 
\item  each $\omega^q_\alpha (t)$ localizes at $x.$  
\end{enumerate}
\end{theorem} 
 A  proof of Theorem \ref{T3.3} can  probably be derived from \cite{Sh} Theorem 1.1, 
 or from  \cite {CFKS}, section 11.5  formula (11.36). 
 A proof on the lines of the proof of Theorem \ref{T3.2}  is suggested in Section \ref{S5} and will be detailed in a paper in preparation. 
\vskip .1in 

\section {CW-complex structure associated with $(M,g,f)$ and the geometric complex}\label {S:CWstructure}

For $(M^n, g,f)$ with $(M^n,g)$ a closed Riemannian  and $f:M\to \mathbb R$ a Morse function denote by $\gamma_y(t)$ the trajectory of the vector field $X= -\grad_g f$ with $\gamma_y(0)=0.$  

For $x\in Cr(f)$ the set $W^\pm_x:=\{y\in M\mid \lim_{t\to \pm\infty} \gamma_y(t)=x\}$ is called the stable/unstable set of $x$ and is a submanifold diffeomorphic to $\mathbb R^{n-k}/ \mathbb R^{k}.$ 
 
 One  says that the vector field $X$ above is {\it Morse-Smale} if for any $x,y\in Cr(f)$  the unstable set $W^-_x$ and the stable set  $W^+_y$ are transversal which implies that $\mathcal T(x,y)= (W^-_x\cap W^+_y)/ \mathbb R$ \footnote {$\mathbb R$ acts freely by translation along the flow defined by $-\grad_gf$},
 the space of trajectories from $x$ to $y,$ is a manifold of dimension $\ind (x)- \ind (y)-1.$ 
For any string of critical points $x=y_0, y_1, \cdots,  y_k$ with
$$\ind (y_0) > \ind (y_1)>,\cdots,> \ind (y_k),$$
consider the smooth manifold of dimension  $\ind (y_0)-k,$
$$\mathcal I(y_0,y_1)\times\cdots \mathcal I(y_{k-1},y_k)\times W_{y_k}^-,$$
and the smooth
map $$i_{y_0, y_1,\cdots ,y_k}: \mathcal I(y_0,y_1)\times\cdots
\times \mathcal I(y_{k-1},y_k)\times W_{y_k}^-
\to M,$$
defined by
$i_{y_0, y_1,\cdots ,y_k}(\gamma_1,\cdots, \gamma_k, y):= i_{y_k}(y)$,
where
$i_x: W_x^- \to M$ denotes the inclusion of
$W_x^-$ in $M.$

\begin{theorem}\label {T4.1}\

1)For any critical point $x\in Cr(h)$
the smooth manifold $W_x^-$ has a canonical compactification
$\hat W_x^-$ to a
compact manifold
with corners \footnote {Recall that an $n-$dimensional  manifold $X$ with corners is a paracompact
Hausdorff space
equipped with a
maximal smooth atlas with charts
$\varphi : U\to \varphi(U)\subseteq \Bbb R^n_+$
with $ \Bbb R^n_+ = \{(x_1, x_2,\cdots x_n) | x_i\geq 0\}.$  The
collection of points
of $ X $
which correspond (by some and then  any chart) to points in
$\Bbb R^n$
with exactly $k$ coordinates equal to zero is
a well defined subset of $X,$ called the $k-$boundary, and  denoted by $X_k.$ It
has a structure of a smooth
$(n-k)-$dimensional manifold.
The subset $\partial X = X_1 \cup X_2 \cup \cdots X_n$
is a closed subset which is a topological manifold of dimension $(n-1)$ and the pair
$(X,\partial X)$
is a topological manifold with boundary $\partial X.$ A compact
smooth manifold
with corners 
with interior diffeomorphic to the Euclidean space,
will be called a compact smooth cell and is homeomorphic to the compact unit disc of the same dimension.  
When $\dim W^-_x=3$ in order to establish this result  apparently 
one needs the recently proven  Poincar\'e conjecture.}  
 and the inclusion $i_x$ has a smooth extension
$\hat i_x: \hat W_x^- \to M$ so that :

\noindent (a) the $k-$ boundary is $(\hat W_x^-)_k = \bigsqcup_{(x,y_1,\cdots, y_k)}
\mathcal I(x,y_1)\times \cdots\times
\mathcal T(y_{k-1}, y_k)\times W_{y_k}^-,$

\noindent (b) the restriction of $\hat i_x$ to
$\mathcal I(x,y_1)\times \cdots \times
\mathcal T(y_{k-1}, y_k)\times W_{y_k}^-$ is 
 $i_{x, y_1\cdots, y_k}.$

2) $\hat W_x^-$ is homeomorphic to the compact disc of dimension $\text{index}(x).$
\end{theorem}
A proof of this result is contained in  \cite{BFK}
see also  \cite {BH1}.  As a consequence one has 
\begin{proposition}\label {P4.2}\

If $M$ is a closed manifold then the following holds true:
\begin{enumerate}
\item   for any $x, y$ critical points with $\ind (x)- \ind (y)=1$ the set $\mathcal T(x,y)$ is finite,
\item  the canonical embedding $\iota_x: W^-_x\to M$ extends to a smooth map $\hat \iota_x: \hat W^-_x\to M$ where  $\hat W^-_x$ is a compact smooth manifold with corners (see Theorem \ref{T4.1} above)  which makes $M= \sqcup_{x\in Cr(f)} W^-_x$ a  CW complex with open cells $W^-_x$, $x\in Cr(f),$
\item 
if for any $x\in Cr(f)$ one chooses an orientation $O_x$ of $W^-_x$  then to  any $\gamma \in \mathcal T(x,y),$ with $\ind x-\ind y=1,$ a sign $\epsilon(\gamma) =\pm 1$ can by 
defined \footnote{ (by comparing the orientation  $Q_x$  to the orientation defined  by the tangent to $\gamma$ plus the orientation $O_y$} as explained in \cite {BFK}.
The incidence  $I(x,y)$ of the cell $W^-_x$ and $W^-_y $ 
is exactly $\sum_{\gamma\in \mathcal T(x,y)}  \epsilon(\gamma).$ 
\end{enumerate}
\end{proposition}
Note that changing $f$ into $-f$ one obtains an other CW structure of $M$ whose cels are the {\it stable sets} of the critical points $x$ of $f,$ i.e. of dimension $n- \ind x$ with 
$I^{(-f)}(y,x)=(-)^{q(n-1-q)}I^f(x,y),$  $q= \ind (x).$

\vskip .1in

 {\bf The geometric complex associated with $-\grad_g f$ and integration maps.}
\vskip .1in
Define  $C^k:=Maps (Cr_k (f),\mathbb R)$  with a base provided by the characteristic functions $E_x,$ $x\in Cr_k(f),$ \footnote{ given by   $E_x(z)= \delta_{x,z} , x, z\in Cr_k(f),$} 
and the linear maps 
$\partial ^k: C^k\to C^{k+1}$  defined  by  $$\partial ^k (E_x) = \sum_{x\in Cr_k(f)} I(x,y) E_y,$$  $x\in Cr_{k+1}(f)$, $y\in Cr_k(f).$    

In view of Proposition  \ref{P4.2} one has 
$\partial ^{k+1}\cdot \partial ^k =0.$   
In view of Theorem  \ref{T4.1}  for $\omega\in \Omega^r(M)$ and $\ind x=r$ one has $\int_{W^-_x} \iota^\ast  (\omega)  <\infty$  and then  
one obtains the linear maps $Int^r: \Omega^r(M) \to C^r.$ 

In view of Stokes theorem  and Theorem \ref{T4.1} one has
$\partial ^k\cdot Int ^{k}= Int^{k+1}\cdot d ^k$  which  makes  $Int ^\ast : (\Omega^\ast (M), d^\ast)\to (C^\ast,\partial ^\ast)$ a morphism of cochain complexes which by 
de-Rham theorem is a quasi-isomorphism. 
 In the diagram below all arrows are quasi-isomorphisms (i.e. induce isomorphisms in cohomology).
$$\xymatrix{ (\Omega^\ast_{vs}(M)(t), d^\ast (t))\ar[r]^{\subset} \ar@/^2pc/[rrr]|{Int^r_{vs}(t)}    &(\Omega^\ast (M), d^\ast (t))\ar[r]^{\cdot e^{tf}}  \ar@/_2pc/[rr]|{Int^r(t)}   & (\Omega^\ast (M), d^\ast )\ar[r]^{Int^\ast}& (C^\ast, \partial^\ast)}$$

For  future needs  one considers the scaling  $\mathcal S^\ast (t)$ of the geometric complex, precisely  the isomorphism  $\mathcal S^\ast (t): (C^\ast, \partial ^\ast)\to (C^\ast, \partial ^\ast(t))$ with 
$\mathcal S^q(t): C^q= Maps(Cr_q(f),\mathbb R) \to C^q=Maps(Cr_q(f),\mathbb R)$ defined by 
\begin{equation}\label {E4.1}
\mathcal S^q(t)(E_x):= (\pi/t)^{(n-2q)/4} e^{-tq} E_x
\end{equation} 
 where $E_x $ is the characteristic function of $x\in Cr_q(f)$ and $\partial^q (t)= \mathcal S^{q+1}(t) \cdot \partial ^q \cdot (\mathcal S^q(t)^{-1}.$ Clearly $(C^\ast, \partial^\ast (t))$ is an analytic family  of cochain complexes  in $t\in \mathbb R$.

\section{Sketch of the proof of Theorems \ref{T3.2} and  \ref{T3.3}} 
\label {S5}
\vskip .1in
The proof of Theorems \ref{T3.2}  and \ref{T3.3} is based on the "mathematics of harmonic oscillator",  in our  case of {\it multivariable harmonic oscillators}  cf. Propositions \ref{P5.1} and \ref{P5.2} and  on a 
criterion for detecting a {\it gap in the spectrum}  of a nonnegative self adjoint operator in a Hilbert space
$H$ cf. Lemma \ref{L5.3} below. 
\vskip .1 in

Recall from \cite{JG} that 
the {\it classical harmonic oscillator}, is the operator 
\begin{equation}\label {E5.1}
-\frac{d^2}{ dx^2} + x^2 : L^2(R)\to L^2(R)
\end{equation}
which when considered as an unbounded  self adjoint operator with domain $\mathcal S(R),$ the space of rapidly decaying function,  has the pure spectrum with eigenvalues $\{ 2j+1 , j\in \mathbb Z_{\geq 0}\}$ and corresponding eigenfunction 
$H_j(x) e^{-x^2}$ with $H_j(x)$ the $j-$th  Hermite polynomial.  

This permits to derive the eigenvalues and the eigenfunctions for 
\begin{equation}\label {E5.2}
-\frac {d^2}{dx^2} + a + b x^2 : L^2(R)\to L^2(R).
\end{equation}

We call  the operators of the form 
\begin{equation}\label {E5.3}
\Delta^q  + A + b\sum_i x_i^2
\end{equation} 
on  
$ \Omega^q (\mathbb R^n) = \{ \omega=\sum _I a_I(x_1, x_2, \cdots x_n) dx_I\}$ {\it multivariable  harmonic oscillator}.
Here   
$I= (j_1, j_2, \cdots j_q),$ $1\leq j_1< j_2, \cdots < j_q\leq n,$ $dx_I= dx_{j_1} \wedge \cdots \wedge dx_{j_q}$
with 
\begin{equation*}
\Delta^q \omega= -\sum _I \sum_{1\leq i\leq n} \frac{\partial ^2}{\partial x_i^2} a_I (x_1, \cdots x_n) dx_I ,\ \ 
A\omega= \sum_I \epsilon_I a_I (x_1, \cdots x_n) dx_I, \  \epsilon _I\in \mathbb R,  \ \ 
b\in \mathbb R_{>0}.
\end{equation*}
As an unbounded self adjoint operator on $L^2(\mathcal S^q(R^n))$ with domain $\mathcal S^q(R^n)= \{\omega\in \Omega^q(R^n)\mid  a_I(x_1, x_2, \cdots x_n)\in \mathcal S(\mathbb R^n)\}$  the operator (\ref {E5.3}) is globally elliptic  in the sense of Shubin \cite {Sh2}, and can be decomposed  as tensor product of operators of the form (\ref{E5.2})
hence its eigenvalues and eigenforms can be calculated explicitly.

In this paper we consider the operator $\Delta^{q,k}(t)$ on $L_2(\mathcal S^q(R^n))$  of the form (\ref{E5.3})
with 
$\epsilon_I ^{q,k}= (-n +2k -2q + 4\sharp \{j\in I \mid k+1\leq j \leq q\}) t $, \ 
$b=t^2$ and notice that this is exactly the Witten Laplacian for the function  $f_k:\mathbb R^n\to \mathbb R$
\begin{equation}\label {E5.4}
f_k(x_1, \cdots x_n)= c- 1/2 \sum_{1\leq i \leq k} x^2_i + 1/2 \sum_{k+1\leq i \leq n} x^2_i.
\end{equation}

For $I= (j_1, j_2, \cdots j_q)$ with $1\leq j_1< j_2< \cdots j_q\leq n$ and $P=(p_1, p_2, \cdots p_n)$ with $p_i\in \mathbb Z_{\geq 0}$ define 
\begin{equation}\label {E5.5}
o^k(I,P)=\sum p_i + q+k -2\sharp \{j\in I\mid 1\leq j\leq k\}. 
\end{equation}
Since $\Delta^{q,k}(t)$ is of the form (\ref{E5.3}), one can calculate its eigenvalues and eigenforms using the eigenvalues and the eigenfunctions of (\ref{E5.2})
and obtain 
\begin{proposition}\label {P5.1} (cf \cite{BFK4}) \

The eigenvalues $\lambda^{q,k}_\beta(t)$ and their corresponding eigenforms $\omega^{q,k}_\beta(t)(x)$ abbreviated $\omega^{q,k}_\beta(t)$ of the globally elliptic operator $\Delta^{q,k}(t)$ are indexed by pairs $ \beta=(I,P)
\in \mathcal I^q\times (\mathbb Z_{\geq 0})^n$ with $o^k(I,P) \geq 0$ 
and are exactly 
\begin{enumerate}  
\item $\lambda^{q,k}_\beta (t)= 2t \cdot o^k(\beta)$
\item $\omega^{q,k}_\beta(t)= H_{p_1}(\sqrt t x_1)\cdots H_{p_n}(\sqrt t x_n) e^{-t |x|^2 /2} dx_I$ .

%

 \end{enumerate}
 \end{proposition}
In view of the above one has.
\begin{proposition}\label {P5.2}\
For any $q$ integer between $0$ and $n$ and $N\in \mathbb Z_{\geq 0}$ 
\begin{enumerate}
\item the collection of pairs $(I,P)$ with $o^k(I,P)= N,$ 
is finite, 
\item  if $N=0$ then  $k=q,$ hence $\ker \Delta^q(t)= \ker\Delta^{q,q}(t),$ and there is only one pair  $(I_0, P_0)$, $I_0= (1,2,\cdots q)$ and $P_0= (0,0,\cdots, 0)$ 
s.t   $ \lambda^q(t):=  \lambda^{q,q}_{I_0, P_0}=0$ with the corresponding eigenform  
$$\omega^q(t):= \omega^{q,q}_{I_0, P_0}(t)= 
(t/{\pi})^{n/4} e^{-t\sum_i x_i^2/2} dx_1\wedge\cdots
 \wedge dx_q.$$
 \end{enumerate}
\end{proposition}

\vskip .1in \hskip 5in $\blacksquare$ 

The main criterion to recognize a gap in the spectrum  is the following lemma whose proof can be found in  \cite {BFK3} Lemma 1.2. 
 \begin {lemma} \label {L5.3}
 Let $A: H \to H$ be a densely defined (not
 necessary bounded )
 self adjoint nonnegative  operator in a Hilbert space $(H,<,>)$
and  $a,b$
two real numbers so that $0<a <b < \infty.$  Suppose that there exist
two closed subspaces $H_1$ and $H_2$ of $H$ with $H_1\cap H_2 =0$
and $H_1 + H_2 =H$ such that
 \newline (1) $<Ax_1, x_1> \ \leq \ a||x_1||^2$ for any $x_1\in H_1,$
 \newline (2) $<Ax_2, x_2> \ \geq \ b||x_2||^2$ for any $x_2\in H_2.$

 Then $spect A\bigcap (a,b)= \emptyset.$
 \end{lemma}

 Lemma \ref{L5.3} will be applied to $H = L_2(\Omega^q(M)),$  
$A=\Delta^q(t)$ for the proof of Theorem \ref{T3.2} and $A=1/t \cdot \Delta^q(t), t>0$ for  the proof of Theorem \ref{T3.3} with $H_1$ described  explicitly 
and $H_2:= H_1^\perp.$ 

For 
Theorem \ref{T3.2} the subspace $H_1$ will be of dimension $c_q= \sharp Cr_q(f).$ 
 To construct $H_1$ we first construct for any $\eta >0$  the form  $ \tilde \omega^{q,\eta}(t)\in \Omega^q_c(\mathbb R^n)$ 
  with support in the ball
$\{|x|\leq \eta\},$ $|x|= \sqrt{\sum_i x_i^2},$ which agrees with $\omega^q(t)$ given above on the ball $\{|x|\leq \eta/2\}$ 
and satisfies
$< \tilde\omega^{q,\eta}(t), \tilde\omega^{q,\eta}(t)> =1$\footnote{The scalar product $\langle\  , \ \rangle$ 
 on $\mathcal S^q(\Bbb R^n)$ is induced by the Euclidean metric on $\mathbb R^n.$} 
 \ and place this form as a form $\overline \omega^q_y(t)$ on $M$ in a neighborhood of each critical point $y\in Cr_k(f).$

Choose a smooth  function $\gamma_\eta(u),$
$\eta\in (0,\infty), \ u\in \Bbb R$
s.t.  

\begin{equation} \label {E5}\gamma_\eta(u)=\quad
\left\{\aligned 1\ \text {if} \ & u\leq \eta/2 \\ 0 \ \text{if}\  &u
>\eta\endaligned \right\}.
\end{equation} and define 
first $\tilde \gamma_\eta : \mathbb R^n\to \mathbb R$ by $\tilde \gamma_\eta(x) = \gamma_{\eta}(|x|)$
and then 
$\tilde \omega^{q,\eta}(t) \in \Omega^q_c(\Bbb R^n) $  by
$$\tilde\omega^{q,\eta}(t) = (1/ \beta_q(t)) \cdot \tilde \gamma_\eta \cdot
\omega^{q}(t)$$ with  
\begin{equation}\label {E6} \beta_q(t)= (t/{\pi})^{n/4}( \int_{\Bbb R^n}
\gamma_\eta^2(|x|) e^{-t |x|^2} dx_1\cdots dx_n)^{1/2}.
\end{equation}

The following proposition can be
obtained
by elementary calculations in coordinates in view of the explicit
formula of
$\Delta^{q,k}(t),$ its eigenforms and of Propositions \ref{P5.1} and \ref{P5.2}(cf. \cite {BFKM}, 
Appendix 2)\footnote {A conceptual derivation of this proposition can be also obtained in view of  Propositions \ref{P5.1} and \ref{P5.2}}.

\begin {proposition}\label{P5.4} 
For 
a fixed $r\in \Bbb Z_{\geq 0}$ 
there
exist positive constants $C,C',C'', T_0,$ and $\epsilon_0$ so that $t >T_0$ and
$\epsilon <\epsilon_0$ imply

(1) $ |\frac{\partial^{|\alpha|}}{\partial x_1^{\alpha_1}\cdots
\partial x_n^{\alpha_n}}\Delta^{q,q}(t)
\tilde\omega^{q,\epsilon}(t)(x)|\leq Ce^{-C't}$ for any
$x\in \Bbb R^n$ and multiindex $\alpha= (\alpha_1,\cdots,\alpha_n)$
with
$|\alpha|=\alpha_1+\cdots+ \alpha_n \leq  r$

(2) $<\Delta^{q,k}(t)\tilde\omega^{q,\epsilon}(t)(x),
\tilde \omega^{q,\epsilon}(t)(x)>\  \geq\  2t|q-k|$

(3) If $\omega \perp \tilde\omega^{q,\epsilon}(t)$  with respect to
the scalar
product $<.,.>$
then $$<\Delta^{q,q}\omega,\omega>\  \geq \ C'' t||\omega||^2.$$
\end{proposition}
In order to place $\tilde \omega^{q,\eta}(t)$ on $M$ one proceeds as follows.

By hypothesis, in a neighborhood $U_y$ of each  $y\in Cr_q(f),$  one chooses  Morse charts $\varphi_y: U_y\to \mathbb R^n$ s.t. 
in these coordinates 
$f$ has the form (\ref {E5.4})
and  the metric $g$ given by $g_{i,}= \delta_{ij}.$
Choose
$\epsilon >0$ small enough s.t  the inverse image of the open discs  of radius slightly larger  than $2\epsilon$  are  disjoint neighborhoods of the critical points and
 consider the smooth forms $\overline{\omega}^q_{y}(t) \in \Omega^q(M)$
defined by
\begin{equation} \label {E5.6} \overline{\omega}^q_{y}(t) |_{M\setminus \varphi_y^{-1}
(D_{2\epsilon})}:=0, \ \
\overline{\omega}^q_y(t)|_{\varphi_y^{-1}(D_{2\epsilon})}:=
\varphi_y^*(\tilde\omega^{q,\epsilon}(t)).
\end{equation}

For any given $t$ the forms $\overline{\omega}^q_y(t) \in\Omega^q(M), \ y\in Cr_q(f),$ are
orthonormal. Indeed,
if $y,z\in Cr_q(f),$ $y\ne z,$ \  $\overline{\omega}^q_y(t)$ and
$\overline{\omega}^q_z(t)$
have disjoint support, hence are orthogonal, and because the support
of $\overline{\omega}^q_y (t)$
is contained in charts as specified above satisfy 
$< \overline{\omega}^q_y(t), \overline{\omega}^q_y(t)> = 1.$ 

One considers the linear map 
$J^q(t): C^q \to \Omega^q(M)$
defined by
$J^q(t)(E_y) = \overline{\omega}^q_y,(t) .$ 
 Clearly $J_q(t)$ is an isometry, thus  injective.
\vskip .2in

{\bf Proof of Theorems \ref{T3.2}:} (sketch). Take 
$H_1:= J^q(t)(C^q)$ and
$H_2= H_1^{\perp},$
let $T_0, C, C', C''$ be given by Proposition \ref{P5.4} and define
\begin{equation}
\begin{aligned}
C_1:=&\inf _{z\in M'} || grad_g f(z)||,\\
C_2:=& \sup_{z\in M} ||(L_{ -grad_g f}(z) +\mathcal L_ {-grad_g f}(z)
||
\end{aligned}
\end{equation}
for $ M'= M\setminus \bigcup_{y\in Cr_q(\alpha)}
 \varphi_y^{-1}(D_\epsilon).$
 
Here $||grad_g f(z) ||$ resp.
$||(L_{-grad_g f}(z) +\mathcal L_ {-grad_g f}(z)||$ denote
the norm of
$ grad_g f (z)$ $\in T_z(M)$ resp. of the linear map
$(L_{-grad_g f} +\mathcal L_ {-grad_g f})(z):
 \Lambda^qT^\sharp_z(M) \to
 \Lambda^qT^\sharp_z(M)$
with respect to the scalar product  induced in $T^\sharp_z(M),$ the dual to $T_z(M).$ Note that if $X$ is a vector field then $L_X +\mathcal L_X$ is a
 differential operator of order zero,
hence an endomorphism of the bundle $\Lambda^qT^\sharp (M)\to M,$ whose space of smooth sections is $\Omega^q(M).$

We can use the constants
$T_0, C, C', C'', C_1, C_2$ to construct $C'''$ and $\epsilon_1$ so
that for $t> T_0$
and $\epsilon <\epsilon_1,$
one has
$< \Delta^q(t) \omega, \omega > \geq C_3t <\omega, \omega>$ for any
$\omega \in H_2$
 (cf. \cite{BFKM}, page 808-810).

One  applies Lemma \ref{L5.3} for
$a= Ce^{-C't}, b =C''' t$ with $t>T_0.$
This concludes the first  part of Theorem \ref{T3.2}.
and also establishes that $c_q$ is larger or equal to the number of eigenvalues $\lambda^q_\alpha(t) \in [0, C e^{-C' t})$ for $t>T_0.$

In order to check that $c_q$ is smaller or equal to this number  
 let $Q^q(t)$ denote the orthogonal projection in $H$ on the
 span of  the  eigenvectors
 corresponding to the eigenvalues smaller than $1.$ In view of the
 ellipticity of
 $\Delta^q(t)$ all
 these eigenvectors are smooth $q-$forms.  
 
 Let $\mathcal I^q(t): =Q^q(t)\cdot J^q(t).$
  By  decreasing $\epsilon $ and increasing $T_0$ 
  one can insure (via rather technical estimates, cf. Proposition 5.4 in \cite {BFKM})  
 that $Q^q(t)\cdot J^q(t)$ is so closed to $J^q(t)$ that it is also  injective for $t$ very large; hence  $c_q$ is smaller or equal to the number of eigenvalues  $\lambda^q_\alpha(t) \in [0,1]$ which for $t$  large enough equals to the number  eigenvalues  $\lambda^q_\alpha(t) \in [0,C e^{-C' t}].$  
 
 \vskip .1in
{\bf About the proof of Theorem \ref{T3.3}}

A proof of Theorem \ref{T3.3} can be obtained in similar way using Lemma \ref {L5.3} applied to the same Hilbert space $H$, operator $A = (1/t )\cdot \Delta^q(t),$ and $H_1$ the finite dimensional vector space spanned by the forms $\overline \omega^q_{\alpha}(t),$  $\alpha= (x, \beta)= (x,I, P),$  with $x\in Cr(f)$ and $0\leq o(\alpha)= o^{\ind x}(\beta) \leq N$ derived from 
$\tilde \omega^{q, \ind(x)}_{I,P}(t)$  in a similar manner the forms $\overline \omega^q_y(t)$ were derived from the forms $\tilde \omega^q(t).$ 
As before one choses  $H_2:= H_1^{\perp}$ and one produces 
the constants $C_1, C_2,   \epsilon_1, \epsilon_2, T$ such that for $t>T$  
the hypotheses of Lemma \ref {L5.3} are satisfied for $t>T$ with $a= 2N + C_1\cdot t^{-\epsilon_1}, b= 2N+2 - C_2\cdot  t^{-\epsilon_2}.$
One establishes first an analogue of Proposition \ref {P5.4} where the forms $\omega ^{q,\epsilon}(t)$ are replaced by  $\omega^{q,k;\epsilon}_{I,P} (t)$ with $o^k(I,P)\leq 2N$ and  inequalities (1), (2), (3) able to  insure that for $t>T,$  strictly below the line $x(t)=t, y(t)=t (2k+1)$ in the complex plane, there are exactly $\sharp \{(i,P)\mid o(I,p)= N \}$ eigenvalue branches  and all other eigenvalue branches are strictly above this line. Moreover each  eigenvalue branch (below this  line)  has $\lim _{t\to \infty} \lambda^q(t)/t$ convergent to $2k', k'\leq N.$  
Details will be provided in a paper in preparation.
So far we are unable to decide if there exists  branches with $\lambda^q(t)/t$  unbounded.
 
\section{Virtually small 
spectral package of $(M,g,f)$}

For each $q$ denote by $\mathcal A^q_{vs}, \mathcal A^q_{la}, \mathcal A^q_{vs,0}, \mathcal A^q_{vs,+}$ the subsets of $\mathcal A^q$ defined by
\begin{enumerate}
\item $\mathcal A^q_{vs}:= \{\alpha\in A^q\mid  o(\alpha)=0\},$
\item $\mathcal A^q_{la}:= \{\alpha\in A^q\mid  o(\alpha)\ne0\},$
\item $\mathcal A^q_{vs,0}:= \{\alpha\in A^q_{vs}\mid  \lambda^q_\alpha(t) = 0 \},$
\item $\mathcal A^q_{vs, +}:= \{\alpha\in A^q_{vs}\mid  \lambda^q_\alpha (t) \ne 0 \}.$
\end{enumerate}
Note that $\mathcal A^q_{vs}$ is the same as  the cluster $\mathcal A^q(0)$ considered in the introduction.  
Clearly $\mathcal A^q_{vs} = \mathcal A^q_{vs,0} \sqcup \mathcal A^q_{vs,+}.$

Let $\Omega^q(M)_{vs}(t),\  \Omega^q(M)_{vs,0}(t), \ \Omega^q(M)_{vs, +}(t)$ be the span of $\omega^q_\alpha(t)$ with 
$\alpha\in \mathcal A^q_{vs},  \mathcal A^q_{vs,0},$ $\mathcal A^q_{vs,+},$ respectively; let  $\Omega^q(M)_{la}(t)$ be the 
orthogonal complement of $\Omega^q(M)_{vs}(t)$ in $\Omega^q(M).$  
 The first three  are finite dimensional subspaces of $\Omega^q(M)$  of dimension 
 $c_q,$
  $\beta_q(M),$ $c_q - \beta_q(M)$  and all four are equipped with a scalar product, the restriction of the scalar product on $\Omega^q(M).$ 
All these subspaces are preserved by $d^\ast (t)$ and left invariant by $\Delta^q(t)$  and in view of Hodge decomposition  for  $\Delta^q(t),$
$\Omega^q(M)= \Omega^q_{vs,0}(M)(t) \oplus  \Omega^q_{vs,+}(M)(t)\oplus \Omega^q_{la}(M)(t).$ 

In particular, one has 
 \  $(\Omega^\ast_{vs} (M)(t), d^\ast(t)),  (\Omega^\ast_{la}  (M)(t), d^\ast(t)) \subset (\Omega^\ast (M), d^\ast(t))$
  and 
 \newline $(\Omega^\ast_{vs,0}  (M)(t), 0), \  (\Omega^\ast_{vs,+}  (M)(t), d^\ast(t)) \subset (\Omega^\ast (M)(t)_{vs}, d^\ast(t))$  inclusions of sub complexes  and   
 each is an analytic family in $t$ of cochain complexes. One has   
\begin{equation} \label {E6.1}
\begin {aligned} 
(\Omega^\ast_{vs,0}  (M)(t), 0)\oplus  (\Omega^\ast_{vs,+}  (M)(t), d^\ast(t)) = &(\Omega^\ast (M)(t)_{vs}, d^\ast(t))\\
(\Omega^\ast_{vs}  (M)(t), d^\ast(t))\oplus  (\Omega^\ast_{la}  (M)(t), d^\ast(t)) = &(\Omega^\ast (M), d^\ast(t)).
\end{aligned}
\end{equation}
Consider the compositions 
 $$\xymatrix {&(\Omega^\ast_{vs}(M)(t), d^\ast(t))\ar[r]^{\subset}\ar@/^3pc/[rrr]|{L(t)^\ast}  \ar@/^2pc/[rr]|{Int^\ast_{vs}(t)} & (\Omega^\ast (M), d^\ast(t))\ar[r]^{Int^\ast(t)}  &(C^\ast,\partial ^\ast) \ar [r]^{S^\ast(t)} &(C^\ast, \partial ^\ast(t))}$$
with $L^q(t)= S^q(t)\cdot Int^q_{vs}(t)$ and $S^q(t)$ given by (\ref{E4.1}).
All these complexes are equipped with scalar product; 
on $C^q$ one considers the scalar product which makes the base $E_x\in C^q$
orthonormal. 

If $\varphi: (V, \langle, \rangle_V)\to (W, \langle, \rangle_W)$ is a linear map between two finite dimensional  vector spaces with scalar products define  
$Vol(\varphi):=(\det (\varphi^\sharp \cdot \varphi))^{1/2}$  which is different from zero iff $\varphi$ is injective. Define $$a^q(t):= Vol (Int_{vs}^q(t))$$  
analytic function in $t$ which has holomorphic extension to a neighborhood of $\mathbb R$ in $\mathbb C.$
  
\begin{theorem}\label {T6.1}\

 If the vector field $(-\grad_g f)$ is Morse-Smale  and the metric $g$ is flat near the critical points then there exists $T$ s.t. for $t> T$ the following holds true.
\begin{enumerate} 
\item $Int^q_{vs}(t)$ is an isomorphism of cochain complexes, 
\item there exists a family of isometries
$R^q(t): C^q \to \Omega^q(M)_{vs}(t)$ of finite dimensional
vector spaces so that
$ L^q(t)\cdot R^q(t)  = Id + O(1/t),$
hence $ L^q(t)$ is an $O(1/t)-$isometry. 
\end{enumerate}
\end{theorem}

Because of item (1) $a^q(t)$ is an analytic function with finitely many zeros  and then
\newline $a(t):=\prod a^q(t)^{(-1^q}$ is a priory a meromorphic function with finitely many zero and poles. 
\vskip .2in

{\bf About  the proof of Theorems \ref{T6.1}}

Item (1) follows from item (2) in view of the definition of $L^q(t).$ 
The isometry $R^q(t)$ is given by :
\begin{equation} 
R^q(t):= \mathcal I^q(t)\cdot  (\mathcal I^q(t)^{\sharp} \cdot \mathcal I^q(t))^{-1/2},\ \mathcal I^q(t)= Q^q(t)\cdot J^q(t)
\end{equation}
where $\mathcal I^q(t)^{\sharp}$ denotes  the adjoint of $\mathcal I^q(t)$.  Note that $\mathcal I^q(t),$ for $t>T,$ is a linear bijective map between f.d. vector spaces with scalar product.

The verification of item (2)  involves a number of estimates and is quite technical. It is a particular case of Theorem 5.5 item 5. in \cite{BFKM}.   
The result can be also recovered  from the work of Helffer an Sj\"ostrandt \cite {HS} but  the proof in \cite {BFKM} is different.

The following result is due to Y. Lee and the present author, cf. \cite{BL}.  
\begin {theorem}   \label{T6.2}\

If $-\grad_f$ is Morse-Smale and the metric $g$ is flat near critical points then the following holds. 
\begin{enumerate}
\item
The meromorphic function $a(t) =\prod (a^q(t))^{(-1)^q} $ has no zero and no poles and has a holomorphic extension to a neighborhood of $\mathbb R$ in $\mathbb C.$
\item  If $M^n$ is a closed odd dimensional manifold  then 
$$\log Tor (M)= 1/2 \sum _q   (-1)^{q+1} q (\sum_{\alpha\in \mathcal A^q_{vs,+}} \ln \lambda^q_\alpha (0))  + \ln a^f (0)  - \sum (-1)^i\ln V^i.$$
\end{enumerate}
\end{theorem}

\begin{proof}
Observe  
\begin{enumerate}
\item  
If  $\varphi (t):(V(t), \langle,\rangle_{V(t)})  \to  (W(t), \langle,\rangle_{W(t)})$ 
is continuous/analytic family of isomorphisms between  finite dimensional vector spaces  equipped with scalar products \footnote{for example $V(t)$ resp.$W(t)$ appear as images in $\mathcal V$ resp.$\mathcal W,$  of an analytic/continuous family of bounded projectors $P(t): \mathcal V \to \mathcal V$ resp. $Q(t): \mathcal W\to \mathcal W$ for $\mathcal V$ resp  $\mathcal W$ topological vector spaces; this give meaning to "analytic family"} 
then the function $Vol(\varphi (t))$ is continuous/analytic in $t.$
 \item For a cochain complex $\mathcal C= (C^\ast, d^\ast)$ of finite dimensional vector spaces equipped with scalar products
$$ \mathcal C: \  \xymatrix {0\ar[r] &(C^0, \langle \ \rangle_0)\ar[r]^{d^0} &(C^1, \langle\  \rangle_1)\ar[r]^{d^1}] &\cdots &(C^n, \langle \ \rangle_n)\ar[r]&0}$$ 
one denotes by $\Delta^q_{\mathcal C}:= \delta^{q+1}\cdot d^q + d^{q-1}\cdot\delta^q,$ \  $\delta $ the adjoint of $d,$ and by $\det' \Delta_C^q\ne 0$ be the product of nonzero eigenvalues  of $\Delta_C^q.$ The product  $$T(\mathcal C):= \prod  ({\det}' \Delta_C^q)^{(-1)^ q q/2}$$  is referred to as the torsion of $\mathcal C.$
For a continuous/analytic family of cochain complexes $\mathcal C(t)= (C^\ast(t), d^\ast(t)) $  such that $\dim C^q(t)$ and $\dim H^q(\mathcal C(t))$ are constant in $t$ 
for any $q$  the function $T(\mathcal C(t))$ is continuous/analytic in $t.$
\newline The verifications of (1) and (2) are straightforward from definitions.

\item Suppose that $\varphi  : \mathcal C_1\to \mathcal C_2,$  is a morphism of cochain complexes of finite dimensional vector spaces with scalar products,. $\mathcal C_{i}= (C^\ast_{i}, d^\ast _{i}), i=1,2,$ \  $\varphi = \{ \varphi^q : C_1^q \to  C_2^q \}.$ 
Suppose that for any $q,$ \ $\varphi^q$ is an isomorphism. Then  $\varphi$ induces 
the isomorphism $H^q(\varphi): H^q(\mathcal C_1)\to H^q(\mathcal C_2)$ between vector spaces equipped  with induced scalar product.   Let 
$$Vol (\varphi): =\prod (vol (\varphi^q))^{(-1)^q}$$ and $$Vol (H(\varphi)): = \prod vol (H^q(\varphi))^{(-1)^q} .$$ As verified in  \cite {BFK2} Proposition 2.5 one has 
 
\begin{equation}\label{E6.3}
T(\mathcal C_2) / T(\mathcal  C_1)= Vol(H (\varphi)) / Vol(\varphi).
\end{equation}
\item  For a continuous/analytic family of isomorphisms $\varphi (t): \mathcal C_1(t)\to \mathcal C_2(t),$\  $t\in \mathbb R,$ with $\dim C^q_1(t)= \dim C^q_2(t)$ and 
$\dim H^q(C_1(t))= \dim H^q(C_2 (t)$  constant in $t,$ the real-valued functions  
$T(C_1(t))$,  $T(C_1(t),$  $Vol (\varphi(t)) , Vol (H(\varphi(t))) $ are nonzero and continuous/analytic. 
\end{enumerate}

We consider  to $\varphi (t)= Int^\ast (t) : (\Omega^\ast_{vs}(M)(t), d^*(t) \to (C^\ast,\partial ^\ast)$  with $\ast= 0,1, \cdots, \dim M.$  
In view of (4) the function $$\frac{T(\Omega^\ast_{vs}(M)(t), d^\ast(t))} {T(C^\ast,\partial ^\ast)} \cdot Vol (H(\varphi(t))$$ is a strictly positive analytic  function and in view of (3) agrees with $a(t)$ for all  $t$ but the finite collection which might be a zero or a  pole for $a(t).$  Hence the  meromorphic function $a(t)$ has no zeros and no poles.
This establishes item (1) in Theorem \ref{T6.2}. Together with  (\ref{E6.3}) it also implies
$$\frac{T(\Omega^\ast_{vs}(M)(t), d^\ast(t))} {a(t)} \cdot Vol (H(\varphi(t))= T(C^\ast,\partial^\ast).$$
Evaluation at $t=0$ combined with the observation that $\mathbb Tor(M)= T(C^\ast,\partial^\ast)$    implies  
$$\frac{T(\Omega^\ast_{vs}(M), d^\ast)} {a(0)} \cdot Vol (H(\varphi(0)))= \mathbb Tor(M).$$ Taking  "$\ln$" one derives item (2) in Theorem \ref{T6.2}.

\end{proof}
\section{ Conjecture,  Questions, Problems}

{\bf Conjecture 1.} \
{\it The set $\mathcal A^q(\infty)$ is empty.}

{\bf Conjecture 2.} \   
{\it For a generic family of pairs $(g,f)$ the positive eigenvalue - branches  and therefore their corresponding eigenform - branches have  multiplicity one .
 
A  stronger  version of this conjecture asks for the same conclusion (multiplicity one for eigenvalue branches) for a generic set of smooth functions $f$ given any metric $g.$}    

In view of \cite {B2}, which implies that  for generic continuous $f,$ $H^r(M;\mathbb R)$ in the presence of a scalar product has a canonical orthonormal base determined by $f$,  a positive answer to Conjecture 1 implies that for $f$ generic with  this base completed by $\omega^q_\alpha (0), \alpha \in \mathcal A^q\setminus \mathcal A^q_{vs,0},$ one obtains a canonical  orthonormal  Hilbert space base for $L_2(\Omega^r(M)).$  This can be used to reduce  PDE problems involving differential forms to  "manipulation of sequences ", as it is the  case of the torus equipped with the flat metric (via Fourier series theory).
\vskip .1in
 {\bf Conjecture  3.} {\it All numbers $a^q \ne 0.$} 

If Conjecture 3 holds true, not only the cohomology has a canonical realization by differential forms (the harmonic forms),  as Hodge theorem in Riemannian geometry states,  but the entire geometric complex $(C^\ast, \partial ^\ast)$ associated to $(g,f),$ at least in the case the hypotheses in Theorem \ref {T6.1} are satisfied, can be canonically realized as a subcomplex $(\Omega^\ast_{vs}(M), d^\ast)$ of  
the de-Rham  complex.  This provides  a substantial and consequential generalization of Hodge theorem.   
\vskip .1in

{\bf Question 4.} {\it We expect that for a fixed metric $g$ (under mild hypotheses) the virtually small spectral package $\{\lambda^q_\alpha(0), \omega^q_\alpha(0), \alpha\in \mathcal A^q_{vs}\},$ 
  depends on the Morse function $f$ only up to homotopy by Morse functions.}.

Such robustness conclusion is necessary in order  to calculate  the virtually small spectral package with arbitrary accuracy by effective numerical methods. 
\vskip .1in

{\bf Question 5} {\it Can {\bf Observation} in Section 1 be used to determine  $\mathcal A_{vs}^q$ inside $\mathcal A^q$?} 
\vskip .1in

{\bf Problem 1.} {\it The works of Cheeger and Buser  cf. \cite {JCH} and \cite{PB} provide bounds  for the first nonzero eigenvalue of $\Delta^0.$ 
One hopes to extend this to the virtually small spectral package and also derive bounds for the eigenvalues of the virtually small spectral package.}
\vskip .1in
The following two problems concern the spectral package  of a Riemannian manifold  $(M,g).$  
\vskip .1in
{\bf Problem 2:} {\it Find  a spectral description of the numbers $V^r$.}

 For $V^n$ we have the famous {\it Weyl law} which evaluates the volume of $M$ in terms of the growth 
of the number of eigenvalues of 
the Laplacian $\Delta ^0$  the same as of $\Delta^n.$ 
\vskip .1in

{\bf Problem 3.} {\it A result of K.Uhlenbeck \cite {U} shows that for a generic Riemannian metric $g$ on a connected closed Riemannian manifold the  eigenvalues of the Laplacian $\Delta^0$ and then of $\Delta^n$ are all simple. hen for a generic metric on an oriented $2-$dimensional manifold the multiplicity of nonzero eigenvalues of $\Delta^q$is one for $q=0,2$ and $2$ for $q=1.$ Clarify the multiplicity of nonzero eigenvalues of $\Delta^q$ for a generic metric on a closed manifold of arbitrary dimension.}

\vskip .2in
{\small
Department of Mathematics,

Ohio State University, 

Columbus, OH 43210,USA}


\begin{thebibliography}{99}

\bibitem{BZ} 
 J. P. Bismut, W. Zhang,
{\it  An extension of a theorem by Cheeger and M\"uller}
Ast\'erisque 205 (1992), 1-209

\bibitem{B2} 
D. Burghelea
{\it 
A refinement of Betti numbers in the presence of a continuous function. ( I )}
Algebr.Geom.Topol. 17(4) 2051-2080 (2017) 
  arXiv:1501.01012

\bibitem{BFK} 
D. Burghelea, L.Friedlander and T. Kappeler, 
{\it 
On the Space of Trajectories of a Generic Vector Field
}  Analele Universitaii de Vest din Timisoara, seria matematica-informatica, vol XLVIII Fasc 1,2, 2010 oo 45-126     arXiv:1101.0778 

\bibitem{BFK2} 
D. Burghelea, L.Friedlander and T. Kappeler, 
{\it   Torsion for Manifolds with Boundary and Glueing Formulas}, Math. Nachr. 208 (1999), 31-91. 

\bibitem{BFK3} 
D. Burghelea, L.Friedlander and T. Kappeler, 
{\it   Witten deformation  of analytic torsion and the Reidemeister torsion, Amer.Math.Soc.Transl.(2), 184,1998,23-39}, 

\bibitem{BFK4} 
D. Burghelea, L.Friedlander and T. Kappeler, 
{\it   Deformation  of the de - Rham complex}, book in preparation. 


\bibitem{BFKM}
D. Burghelea, L. Friedlander, T. Kappeler and P. McDonald, {\it Analytic and Reidemeister torsion for representations in finite type Hilbert modules}, Geom. Funct. Anal. 6(1996), 751Ð 859.

\bibitem{BH1}
 D.Burghelea, S. Haller, {\it On the topology and analysis of closed one form. I (Novikov theory revisited)}, Monogr. Enseign. Math. 38(2001), 133Ð175.

\bibitem{BL}
D.Burghelea, Y.Lee, {\it Torsion and  of virtually small spectral package} in preparation.

\bibitem{PB}
Buser, Peter. "A note on the isoperimetric constant". Ann. Sci. École Norm. Sup. Buser, Peter (1982). "A note on the isoperimetric constant". Ann. Sci. École Norm. Sup. (4). 15 (2): 213?230. 

\bibitem{CFKS} H.L.Cycon, R.G.Froesne, W.Kirsh, B.Simon, {\it Schr\"odinger Operators with Applications  to Quantum Mechanics and Global geometry,}  Texts and Monographs in Physics, Springer -Verlag, 1987

\bibitem{Cheeger}
 Jeff Cheeger, {\it Analytic torsion and heat equation}, Annals of Mathematics, 109 (1979), 259-322.

\bibitem {JCH}
Jeff Cheeger, {\it A lower bound for the smallest eigenvalue of the Laplacian} In Gunning, Robert C. (ed.). Problems in analysis (Papers dedicated to Salomon Bochner, 1969). Princeton, N. J.: Princeton Univ. Press. pp. 195?199. 

\bibitem{JG}
J.Glim, A. Jaffe, {Quantum Physics, A functional integral Point of view}  Springer - Verlag, 1981  

\bibitem{SH} S.Haller, {\it Analytic eigenbranches in the semiclassical limit},   arXiv  2001.07154v1

\bibitem{TK} 
T. Kato, {\it Perturbation theory for linear operators}. Second edition. Grundlehren der Mathe-
matischen Wissenschaften 132. Springer-Verlag, Berlin-New York, 1976.

\bibitem{Sh}
M.A. Shubin, {\it Semiclassical asymptotics  on covering manifolds  and Morse inequalities,} Geom. Funct. Anal.  6 (1966) 370-409

\bibitem{Sh2}
M.A. Shubin, {\it Pseudodifferential operators and spectral theory}  second edition, Springer - Verlag, 2001

\bibitem{U}  K.Uhlenbeck, {\it Generic properties of Eigenfunctions,} Amer. J. Math., 98(1976), 1059-1078.

\bibitem{W} 
E.Witten, {\it Supersymmetry and Morse theory}, J. of Diff. Geom. 17 (1982) 661-692.

\bibitem{HS}  B.Helffer, J.Sj\"ostrand,  {\it Puits multiple en mecanique sem-iclassique., IV  Etude du complexe de Witten} Comm. in PDE 10 (1985) pp 245--340. {Generic properties of Eigenfunctions,} Amer. J. Math., 98(1976), 1059-1078.

\end{thebibliography}
\end{document}